\documentclass[11pt,leqno]{amsart}
\usepackage{amsmath,amssymb,amsthm}
\usepackage{hyperref}
\usepackage{amsfonts,amsmath, amsthm, amscd, amssymb, graphicx, color, mathrsfs, stmaryrd}
\usepackage[utf8]{inputenc}
\usepackage{tikz-cd}
\usepackage{adjustbox}

\usepackage{ulem} 

\usepackage{bbm}

\renewcommand{\geq}{\geqslant}
\renewcommand{\leq}{\leqslant}

\newcommand{\co}{\operatorname{co}}

\newcommand{\cconv}{\overline{\co}}

\newtheorem{thm}{Theorem}[section]

\newtheorem{prop}[thm]{Proposition}
\newtheorem{coro}[thm]{Corollary}
\newtheorem{lema}[thm]{Lemma}

\theoremstyle{definition}
\newtheorem{defi}[thm]{Definition}

\newtheorem{rema}[thm]{Remark}

\numberwithin{equation}{section}

\def\fnote#1{\footnote}

\def\ignora#1{}
\def\n3#1{\left\vert  \! \left\vert \! \left\vert \, #1 \, \right\vert \!
  \right\vert \! \right\vert }

\begin{document}

\title[]{Ergodicity and super weak compactness}

\author[G. Grelier]{ Guillaume Grelier }
\author[M. Raja]{ Matías Raja }
\address[G. Grelier]{Universidad de Murcia, Departamento de Matem\'aticas, Campus de Espinardo 30100 Murcia, Spain} \email{g.grelier@um.es}
\address[M. Raja]{Universidad de Murcia, Departamento de Matem\'aticas, Campus de Espinardo 30100 Murcia, Spain} \email{matias@um.es}

\keywords{}

\maketitle

\begin{abstract}
We prove that a closed convex subset of a Banach space is (super-)weakly compact if and only if it is (super)-ergodic. As a consequence we deduce that super weakly compact sets are characterized by the fixed point property for continuous affine mappings. We also prove that the M-(fixed point property for affine isometries) implies the Banach-Saks property.
\end{abstract}

\section{Introduction}

The notion of \textit{super weakly compact} set was introduced in \cite{Raja} for convex sets under the name \textit{finitely dentable} set. The name super weakly compact is used for the first time in \cite{SWC10}.  Since then, this notion has been deeply studied. For example, the second named author proved in \cite{S2WCG} that a Banach space is super weakly compact generated if and only if it admits a strongly uniformly G\^ateaux smooth norm. Another renorming result from \cite{S2WCG}, quoted here as Theorem \ref{unif_convex_renorm}, is one of the main ingredients of this paper since it implies that convex super weakly compact sets can be renormed to have normal structure, which is a notion strongly related to fixed point properties (see \cite{Kirk}).\\

In \cite{SWC}, the authors proved that a closed bounded convex subset $K$ of a Banach space $X$ is super weakly compact if and only is it has the super fixed point property for affine isometries $T:X\to X$ preserving $K$. In this paper, we improve this characterization in two different ways. On the one hand, we prove that a closed convex set $K$ is super weakly compact if and only if it is super-ergodic (Theorem \ref{super-ergodic}), thanks to an adaptation of the classical mean ergodic theorem (Theorem \ref{ergo}). The fixed point properties then follow easily using the same techniques. On the other hand, the main drawback of the caracterization given in \cite{SWC} is that it can exist affine isometries defined on $K$ which can not be extended to the all space. Moreover, the super weak compactness is a localized version of superreflexivity and it is natural to expect that such a characterization only depends on $K$. This is done is Theorem \ref{SWC_FPP}. \\

The structure of the paper is the following. The next section deals with the several notions that will be use throughout this paper. We start recalling the definition of spreading models and we study when the shift of a spreading model has a fixed point when it is restricted to the closed convex hull of its spreading basis (see Proposition \ref{shift_fixedpoint}). Next, we introduce a definition of finite representability for sets that generalizes the usual definition for Banach spaces. Finally, we recall the notion of super weak compactness.
In the third section, we characterize the (super-)weak compactness in terms of (super-)ergodicity (Theorems \ref{weakly_compact} and \ref{super-ergodic}). For that, we establish a version of the mean ergodic theorem (see Theorem \ref{ergo}).
The fixed point properties of super weakly compact sets are obtained in the fourth section as a consequence of their ergodic properties (see Theorem \ref{SWC_FPP}).
In the fifth part, we apply the results established in the previous sections to strongly super weakly compactly generated Banach spaces. In fact, any weakly compact set is super weakly compact in these spaces and we can obtain stronger results.
In the last section, we prove that if all spreading models of a Banach space $X$ have the fixed point property for affine isometries then $X$ is reflexive and in fact has the Banach-Saks property (Theorem \ref{M-FPP}). It follows that the reflexivity strictly lies bewteen the fixed point property and the M-(fixed point property) (see Definition \ref{M-property}).

\section{Basic Notions}

\subsection{Notation}

In general, our notation is standard and follows textbooks such as \cite{BST}. For example, if $X$ is a Banach space, $B_X$ denotes its closed unit ball and the diameter of a bounded set $A\subset X$ is denoted by $\text{diam}(A)$. If $X$ is a Banach space and $\mathcal U$ is an ultrafilter, the ultrapower of $X$ with respect to $\mathcal U$ is denoted by $X^\mathcal U$. We refer the reader to \cite{Heinrich} for more informations about ultrapowers. 

If $(\mathcal P)$ is a property of Banach spaces, we say that a Banach space $X$
has the property super-$(\mathcal P)$ if any ultrapower of $X$ has $(\mathcal P)$. If $(\mathcal P)$ is hereditary, this is equivalent to the fact that any Banach space which is finitely representable in $X$ has $(\mathcal P)$.

\subsection{Spreading model}

In this part, we recall the definition of spreading models initially introduced by Brunel and Sucheston in \cite{spreading}. We refer the reader to \cite{Beauzamy} for a great presentation of spreading models.\\

Let $X$ be a Banach space and let $(x_n)_{n\in\mathbb N}$ be a sequence in $X$. We said that $(x_n)_{n\in\mathbb N}$ is a \textit{good} sequence if $$\lim_{n_1\to\infty}\|a_1x_{n_1}+...+a_1x_{n_k}\|,$$ where $n_1<...<n_k$, exists for all $k\in\mathbb N$ and all $a_1,...,a_k\in\mathbb R$. Using a theorem of Ramsey, one can prove that every bounded sequence has a good subsequence.  If $(x_n)_{n\in\mathbb N}$ is a good sequence, $\|(a_1,...,a_k)\|=\lim_{n_1\to\infty}\|a_1x_{n_1}+...+a_kx_{n_k}\|$, with $n_1<...<n_k$, defines a semi-norm on $c_{00}$. It is easily seen that it defines a norm if and only if $(x_n)_{n\in\mathbb N}$ is not convergent. In this case, the completion $Z$ of $c_{00}$ with this new norm is called \textit{spreading model} of $X$ built on $(x_n)_{n\in\mathbb N}$. We say that $(e_n)_{n\in\mathbb N}$ (where $(e_n)_{n\in\mathbb N}$ is the canonic basis of $c_{00}$) is the fundamental sequence of the spreading model. It can be proved that $Z$ is finitely reprensentable in $X$. More precisely, one has that:
$$\forall\varepsilon>0\ \ \forall N\geq 1\ \ \exists p\in\mathbb N\ \ \forall n_1<...<n_N\ \ \text{with}\ \ n_1\geq p\ \ \forall a_1,...,a_N\in\mathbb R$$
\begin{equation}
(1-\varepsilon)\left\|\sum_{i=1}^N a_ix_{n_i}\right\|\leq\left\|\sum_{i=1}^N a_ie_i\right\|\leq (1+\varepsilon)\left\|\sum_{i=1}^N a_ix_{n_i}\right\|
\label{fr}
\end{equation}
\\

\noindent A non-constant sequence $(e_n)_{n\in\mathbb N}$ of a Banach space $X$ is said to be spreading if $$\left\|\sum_{i=1}^n a_ie_i\right\|=\left\|\sum_{i=1}^n a_ie_{n_i}\right\|$$ for all $k\in\mathbb N$, all $a_1,...,a_k\in\mathbb R$ and all $n_1<...<n_k$. By construction of the norm of a spreading model, its fundamental sequence is spreading. 

Let $Z$ be a spreading model of a Banach space $X$ built on a bounded good sequence $(x_n)_n$ with fundamental sequence $(e_n)_n$. Following ideas of Brunel and Sucheston (see \cite{Brunel}), we can define a linear isometry $T:c_{00}\to Z$ by $$T\left(\sum_{i\geq 1}a_ie_i\right)=\sum_{i\geq 2}a_{i-1}e_i.$$ Then $T$ extends to a linear isometry from $Z$ to $Z$. It is clear that $T(\cconv\{e_n\}_n)\subset \cconv\{e_n\}_n$. Along this document, we will always refer to that mapping as \textit{the shift of the spreading model} $Z$.

\begin{prop}\label{shift_fixedpoint}
Let $X$ be a Banach space and let $Z$ be a spreading model of $X$ with fundamental sequence $(e_n)_n$. Let $T$ be the shift of $Z$. The following assertions are equivalent:
\begin{enumerate}
    \item[(i)] $T|_{\cconv\{e_n\}_n}$ has a fixed point;
    \item[(ii)] $(e_n)_n$ is weakly convergent.
\end{enumerate}
\end{prop}

\begin{proof}
$(ii)\implies (i)$ Suppose that $(e_n)_n$ weakly converges to $e\in Z$. By weak continuity of $T$, we have that $e_{n+1}=T(e_n)\xrightarrow[n]{w}T(e)$ and since $e_{n+1}\xrightarrow[n]{w}e$ we deduce that $T(e)=e$. Obviously we also have that $e\in\cconv\{e_n\}_n$.

$(i)\implies(ii)$ Let $e\in \cconv\{e_n\}_n$ be a fixed point of $T$. For $n\geq 1$, define $F_n=\overline{\text{span}(e_i)_{i\geq n}}$. It is clear that  $e\in F_\infty:=\bigcap_{n\geq 1} F_n$.

Note that we can suppose that $(e_n)_n$ is a basic sequence. In fact, if $(e_n)_n$ is not a basic sequence then $(e_n)_n$ weakly converges (see Proposition 2 p.17 in \cite{Beauzamy}) and we are done. So let us suppose that $(e_n)_n$ is a basic sequence. Then $F_\infty=\{0\}$ and it follows that $e=0\in \cconv\{e_n\}_n$. Since $(e_n)_n$ is a spreading sequence, this is equivalent to the fact that $e_n\xrightarrow[n]{w}0$.
\end{proof}

\subsection{Finite-representability of sets}

\begin{defi}
Let $X$ and $Y$ be Banach spaces. We say that a set $A\subset X$ is \textit{finitely represented} in a set $B\subset Y$ (in short, $A$ is f.r. in $B$) if for all $\varepsilon>0$ and all $A_0\subset A$ finite linearly independent set, there exist $B_0\subset B$ and an isomorphism $T:\text{span}(A_0)\to\text{span}(B_0)$ such that $T(A_0)\subset B$ and for all $x\in \text{span}(A_0)$ $$(1-\varepsilon)\|x\|\leq\|T(x)\|\leq(1+\varepsilon)\|x\|.$$
\end{defi}

Note that the following definition generalized the usual one for Banach spaces since $X$ is finitely representable in $Y$ (we also write that $X$ is f.r. in $Y$) if and only if $B_X$ is f.r. in $B_Y$.\\

We will need the following lemma:

\begin{lema}
Let $A\subset X$ be a subset of a Banach space $X$ and let $\varepsilon>0$. Then for any finite linearly independent set $x_1,...,x_N$ in $\overline{A}$, there exists a finite linearly independent set $y_1,...,y_N$ in $A$ such that $\|x_k-y_k\|<\varepsilon$ for all $1\leq k\leq N$.
\end{lema}

\begin{proof}
There exist sequences $(y_n^k)_n\subset A$ such that $(y_n^k)_n\xrightarrow[n]{}x_k$ for all $1\leq k\leq N$. Without loss of generality, we can suppose that $\|y_n^k-x_k\|<\varepsilon$ for all $n\in\mathbb N$ and $1\leq k\leq N$. Since $\text{span}\{x_k\}_{1\leq k\leq N}$ is finite dimensional, it is complemented in $X$ and then there exists a bounded onto projection $p:X\to \text{span}\{x_k\}_{1\leq k\leq N}$. By continuity of $p$, we have that $p(y_n^k)_n\xrightarrow[n]{}p(x_k)=x_k$ for all $1\leq k\leq N$. Now using the continuity of the determinant in $\text{span}\{x_k\}_{1\leq k\leq N}$, we deduce that there exists $n_0\in\mathbb N$ such that the family $p(y_{n_0}^1),...,p(y_{n_0}^N)$ is linearly independent for all $n\geq n_0$. The family $y_{n_0}^1,...,y_{n_0}^N$ is linearly independent set and fulfills that $\|x_k-y_{n_0}^k\|<\varepsilon$ for all $1\leq k\leq N$.
\end{proof}

The following result is well-known in the case of Banach spaces and is an adaptation to the finite-representability of sets:

\begin{prop}
Let $A\subset X$ and $B\subset Y$ be subsets of two Banach spaces $X$ and $Y$. Suppose that $A$ can be written $A=\overline{\bigcup_{n=1}^\infty A_n}$ where $(A_n)_n$ is an increasing sequence of sets such that $A_n$ is f.r. in $B$. Then $A$ is f.r. in $B$.

In particular, $A$ is f.r. in $B$ if and only if $\overline{A}$ is f.r. in $B$.
\end{prop}

\begin{proof}
Let $e_1,...,e_N$ be a finite linearly independent set in $\overline{A}$ and let $\varepsilon>0$. Define $E=\text{span}\{x_k\}_{1\leq k\leq N}$. Since $E$ is finite dimensional, there exists $C>0$ such that for all $a_1,...,a_N\in\mathbb R$ $$\frac{1}{C}\max_{1\leq k\leq N}|a_k|\leq\left\|\sum_{k=1}^Na_ke_k\right\|\leq C\max_{1\leq k\leq N}|a_k|.$$ Choose $\nu>0$ such that $(1+\varepsilon)\frac{1+CN\nu}{1-CN\nu}<1+2\varepsilon$. By the previous lemma, there exist a finite linearly independent set $x_1,...,x_N\in\bigcup_{n=1}^\infty A_n$ such that $\|e_k-x_k\|<\varepsilon$ for all $1\leq k\leq N$. Let $n\in\mathbb N$ such that $x_1,...,x_n\in A_n$. Since $A_n$ is f.r. in $B$, there exist $B_0\subset B$ and an isomorphism $T:\text{span}(A_0)\to\text{span}(B_0)$ such that $T(A_0)\subset B$ and $\|T\|\|T^{-1}\|<1+\varepsilon$. Define a linear operator $S:E\to\text{span}(B_0)$ by $S(e_k)=T(x_k)$ for all $1\leq k\leq N$. Take $e=\sum_{k=1}^Na_ke_k\in E$. Note that $$\left\|\sum_{k=1}^Na_ke_k-\sum_{k=1}^Na_kx_k\right\|\leq N\nu\max_{1\leq k\leq N}|a_j|\leq CN\nu\left\|\sum_{k=1}^Na_ke_k\right\|.$$

On one hand, we have that 
\begin{align*}
\|S(e)\|=\left\|T\left(\sum_{k=1}^Na_kx_k\right)\right\|&\leq \|T\|\left\|\sum_{k=1}^Na_ke_k\right\|\\
&\leq \|T\|\left(\left\|\sum_{k=1}^Na_ke_k-\sum_{k=1}^Na_kx_k\right\|+\left\|\sum_{k=1}^Na_ke_k\right\|\right) \\
&\leq\|T\|\|e\|(1+CN\nu)
\end{align*}
and on the other hand
\begin{align*}
\|S(e)\|=\left\|T\left(\sum_{k=1}^Na_kx_k\right)\right\| &\geq \frac{1}{\|T^{-1}\|}\left\|\sum_{k=1}^Na_kx_k\right\| \\
&\geq \frac{1}{\|T^{-1}\|}\left(\left\|\sum_{k=1}^Na_ke_k\right\|-\left\|\sum_{k=1}^Na_ke_k-\sum_{k=1}^Na_kx_k\right\|\right)\\
&\geq \frac{1}{\|T^{-1}\|}\|e\|(1-CN\nu).
\end{align*}
We conclude that $\|S\|\|S^{-1}\|\leq\|T\|\|T^{-1}\|\frac{1+CN\nu}{1-CN\nu}\leq(1+\varepsilon)\frac{1+CN\nu}{1-CN\nu}<1+2\varepsilon$.
\end{proof}

\begin{coro}\label{fr_sm}
Let $Z$ be a spreading model of a Banach space $X$ built on $(x_n)_n$ with spreading sequence $(e_n)_n$. Then $\cconv\{e_n\}_n$ is f.r. in $\co\{x_n\}_n$.
\end{coro}

\begin{proof}
By the previous proposition, it is enough to prove that $\co\{e_1,...,e_p\}$ is f.r. in $\{x_n\}_n$ for all $p\geq 1$. But that follows directly from (\ref{fr}).
\end{proof}

\subsection{Super weak compactness}

\begin{defi}
Let $X$ be a Banach space and let $A\subset X$ be a bounded set.  We say that $A$ is \textit{relatively super-weakly compact} if any set f.r. in $A$ is relatively weakly compact. Furthemore if $A$ is weakly closed, we say that $A$ is \textit{super weakly compact} (in short, SWC).
\end{defi}

This notion is a localized version of superreflexivity in the sense that a Banach space $X$ is superreflexive if and only if $B_X$ is SWC.

We recall that if $A$ is a bounded subset of a Banach space $X$ and if $\mathcal U$ is an ultrafilter on a set $I$, the \textit{ultrapower} of $A$ is defined by $$A^\mathcal U=\{(x_i)_{\mathcal U}\ :\ \forall i\in I\ \ x_i\in A\}\subset X^{\mathcal U}.$$ The two following results are similar to Propositions 6.1 and 6.2 in \cite{Heinrich} and we omit the proofs since they only require minor adjustments.

\begin{prop}\label{fr1}
Let $A$ be a bounded subset of a Banach space $X$. Then $A^\mathcal U$ is f.r. in $A$ for any free ultrafilter $\mathcal U$.
\end{prop}

\begin{prop}\label{fr2}
Let $X$ and $Y$ be Banach spaces. Suppose that $B\subset Y$ is f.r. in $A\subset X$. If $B_0\subset B$ is a linearly independent set, then there exist a free ultrafilter $\mathcal U$ and a linear isometry $T:\text{span}(B_0)\to(\text{span}(A))^{\mathcal U}$ such that $T(B_0)\subset A^{\mathcal U}$.
\end{prop}

We deduce the following caracterization of SWC sets:

\begin{thm}\label{thm.fr}
Let $K$ be a bounded weakly closed subset of a Banach space $X$. The following assertions are equivalent:
\begin{enumerate}
    \item[(i)] $K$ is SWC;
    \item[(ii)] $K^\mathcal U$ is relatively weakly compact for any ultrafilter $\mathcal U$.
\end{enumerate}
\end{thm}

\begin{proof}
$(i)\implies (ii)$ follows from Proposition \ref{fr1}. Now suppose that there exists a set $A$ f.r. in $K$ such that $A$ is not relatively weakly compact. Then $A$ contains a sequence $(x_n)_{n\in\mathbb N}$ which does not admit any weakly convergent subsequence. Without loss of generality, we can consider that $(x_n)_{n\in\mathbb N}$ is a linearly independent family (in fact, $(x_n)_{n\in\mathbb N}$ admits a maximal linearly independent subsequence by Zorn's lemma). By proposition \ref{fr2}, there exist a free ultrafilter $\mathcal U$ and a linear isometry $T:\text{span}(B_0)\to(\text{span}(A))^{\mathcal U}$ such that $T(B_0)\subset A^{\mathcal U}$. It follows that $A^{\mathcal U}$ is not relatively weakly compact since $(T(x_n))_{n\in\mathbb N}$ does not admit any weakly-convergent subsequence and then $K$ is not SWC.
\end{proof}

Let $C$ be a convex subset of Banach space $X$. A function $f:C\to\mathbb R$ is \textit{uniformly convex} if for all $\varepsilon>0$ there exists $\delta>0$ such that $$f\left(\frac{x+y}{2}\right)\leq\frac{f(x)+f(y)}{2}-\delta$$ whenever $x,y\in C$ are such that $\|x-y\|\geq\varepsilon$. The \textit{modulus of convexity} of $f$ is the function $\delta_f:\mathbb R^+\to\mathbb R^+$ defined by $$\delta_f(\varepsilon)=\inf\left\{\frac{f(x)+f(y)}{2}-f\left(\frac{x+y}{2}\right)\ :\ x,y\in C,\ \|x-y\|\geq\varepsilon\right\}.$$  Finally we remember that a set $K\subset X$ \textit{strongly generates} a set $H\subset X$ if for all $\varepsilon>0$ there is $n\in {\mathbb N}$ such that $H \subset nK+\varepsilon B_X$.\\

The following result is Theorem 3.2 in \cite{S2WCG}. It will be used repeatedly in the next parts since it has a strong connexion with normal structure and fixed point properties.

\begin{thm}\label{unif_convex_renorm}
Let $K$ be a SWC absolutely convex subset of a Banach space $X$. Then $X$ admits an equivalent norm $|.|$ such that the restriction of $|.|^2$ to any convex set strongly generated by $K$ is uniformly convex.
\end{thm}

\section{Ergodicity and (super) weak compactness}

\begin{defi}
Let $C$ be a convex subset of a Banach space $X$. We say that an affine function $T:C\to C$ is 
\begin{itemize}
    \item ergodic if the Cesaro mean sequence $\left(\frac{1}{n}\sum_{k=0}^{n-1} T^k(x)\right)_n$ converges for all $x\in C$;
    \item Cesaro equicontinuous if $\left\{\frac{1}{n}\sum_{k=0}^{n-1}T^k\right\}_{n\geq 1}$ is an equicontinuous set.
\end{itemize}
We say that $C$ is \textit{ergodic} if any Cesaro equicontinuous affine function $T:C\to C$ is ergodic. We say that $C$ is \textit{super-ergodic} if any convex set which is f.r. in $C$ is ergodic.
\end{defi}

Note the that the previous definition of ergodicity extends the usual one in a natural way. In fact, remember that a Banach space $X$ is \textit{ergodic} if any Cesaro bounded operator $T:X\to X$ (i.e. $\sup_{n\geq 1}\left\|\frac{1}{n}\sum_{k=0}^{n-1}T^n\right\|<\infty)$ is ergodic (see \cite{ergodic_def}).

\begin{prop}\label{ergo_Banach}
A Banach space $X$ is ergodic if and only if $B_X$ is ergodic.
\end{prop}

\begin{proof}
Suppose that $X$ is ergodic. Let $T:B_X\to B_X$ be a Cesaro equicontinuous affine function. Without loss of generality, we can suppose that $T(0)=0$. Note that $T$ can be extended to $X$ by $T':X\to X$ by $T'(x)=T\left(\frac{x}{\|x\|}\right)\|x\|$. It is easy to prove that $T'$ is linear. Moreover, from the Cesaro equicontinuity of $T$, it is clear that $T'$ is Cesaro bounded. We deduce that $T'$ is ergodic and then $T$ also is.

Now let's suppose that $B_X$ is ergodic. Let $T:X\to X$ such that $T$ is Cesaro bounded. Without loss of generality, we can suppose that $\|T\|\leq 1$. So $T|_{B_X}:B_X\to B_X$ is well-defined, Cesaro equicontinuous and then is ergodic by hypothesis. It follows that $T$ is ergodic.
\end{proof}

We start with an adaptation of the mean ergodic theorem (see Theorem 1.1 p.72 in \cite{ergodic}):

\begin{thm}\label{ergo}
Let $C$ be a bounded convex subset of a Banach space $X$ and let $x,y\in C$. Let $T:C\to C$ be a Cesaro equicontinuous affine function. The following assertions are equivalent:
\begin{enumerate}
    \item[(i)] $Ty=y$ and $y\in\cconv\{T^nx\}_{n\geq 0}$;
    \item[(ii)] $\frac{1}{n}\sum_{k=0}^{n-1} T^k(x)\to y$;
    \item[(iii)] $\frac{1}{n}\sum_{k=0}^{n-1} T^k(x)\xrightarrow[]{w} y$;
    \item[(iv)] $\left(\frac{1}{n}\sum_{k=0}^{n-1} T^k(x)\right)_{n\geq 1}$ has a subsequence that converges weakly to $y$.
\end{enumerate}
\end{thm}

\begin{proof}
For simplicity, we write $S_n=\frac{1}{n}\sum_{k=0}^{n-1} T^k$ for all $n\geq 1$. Obviously we have that $(ii)\implies(iii)\implies(iv)$. Suppose that $(iv)$ holds, i.e. $(S_{\phi(n)}(x))_{n\geq 1}$ weakly converges to $y$ for some stricly increasing function $\phi:\mathbb N\to\mathbb N$. It is clear that $y\in\cconv\{T^nx\}_{n\geq 0}$. Note that for all $n\in\mathbb N$ $$T\left(S_{\phi(n)}(x)\right)=S_{\phi(n)}(x)+\frac{1}{\phi(n)}T^{\phi(n)}(x)-\frac{1}{\phi(n)}x$$ and since $C$ is bounded, we deduce that $T\left(S_{\phi(n)}(x)\right)\xrightarrow[]{w} y$. Moreover $T$ is weakly continuous, so $T\left(S_{\phi(n)}(x)\right)\xrightarrow[]{w} T(y)$. It follows that $Ty=y$.

Now suppose that $(i)$ is true and fix $\varepsilon>0$. Since $T$ is Cesaro equicontinuous, there exists $\eta>0$ such that whenever $z\in C$ fulfills $\|z-y\|<\eta$ then $\|S_n(y)-S_n(z)\|<\varepsilon$ for all $n\in\mathbb N$. There exists a convex combination $\sum_{k=0}^pa_kT^k(x)$ such that $\left\|y-\sum_{k=0}^pa_kT^k(x)\right\|<\eta$. Define an affine function on $C$ by $S=\sum_{k=0}^pa_kT^k$. For all $n>p$, one has that 
\begin{align*}
\left\|S_nSx-S_nx\right\|&=\frac{1}{n}\left\|\sum_{k=0}^{n-1}T^k\left(\sum_{j=0}^pa_jT^j(x)\right)-\sum_{k=0}^{n-1} T^k(x)\right\| \\
&=\frac{1}{n}\left\|\sum_{k=0}^{n-1}\sum_{j=0}^pa_jT^{k+j}(x)-\sum_{k=0}^{n-1} T^k(x)\right\| \\
&=\frac{1}{n}\left\|\sum_{j=0}^pa_j\sum_{k=0}^{n-1}\left(T^{k+j}(x)-\sum_{k=0}^{n-1} T^k(x)\right)\right\| \\
&=\frac{1}{n}\left\|\sum_{j=1}^pa_j\sum_{k=n}^{n-1+j}T^{k}(x)\right\|.
\end{align*}
Since $\lim_n\frac{1}{n}T^n(x)=0$ ($C$ is bounded), we deduce that there exists $N>p$ such that $\left\|S_nSx-S_nx\right\|<\varepsilon$ for all $n\geq N$. It follows that $$\|y-S_nx\|=\|S_ny-S_nx\|\leq\|S_ny-S_nSx\|+\|S_nSx-S_nx\|<\varepsilon+\varepsilon=2\varepsilon$$ for all $n\geq N$ since $\|y-Sx\|<\eta$. We conclude that $y=\lim_nS_nx$.
\end{proof}

We obtain the following caracterization of weak compactness:

\begin{thm}\label{weakly_compact}
Let $C$ be a closed convex subset of a Banach space $X$. The following assertions are equivalent:
\begin{enumerate}
    \item[(i)] $C$ is weakly compact;
    \item[(ii)] any closed convex subset of $C$ is ergodic.
\end{enumerate}
\end{thm}

\begin{proof}
$(i)\implies (ii)$ follows directly from the previous theorem. Now suppose that $C$ is not weakly compact. By Proposition 1 in \cite{Bena}, there exists a basic sequence $(y_n)_n\subset C$ and an affine homeomorphism $\Phi:A\to B$ such that $\Phi(y_n)=e_n$ for all $n\in\mathbb N$ where $$A:=\left\{\sum_{n=1}^\infty a_ny_n\ |\ a_n\geq 0\ \text{and}\ \sum_{n=1}^\infty a_n=1\right\}$$ and $$B:=\left\{\sum_{n=1}^\infty a_ne_n\ |\ a_n\geq 0\ \text{and}\ \sum_{n=1}^\infty a_n=1\right\}$$ with $(e_n)_n$ the canonical basis of $l_1$. Define the bilateral shift on $l_1$ by $$S\left(\sum_{n=1}^\infty a_ne_n\right)=a_2e_1+\sum_{n=1}^\infty a_{2n-1}e_{2n+1}+\sum_{n=2}^\infty a_{2n}e_{2n+2}.$$ Finally we define an affine continous mapping by $T=\Phi^{-1}S\Phi:A\to A$. It is clear that $T$ is Cesaro equicontinuous. In fact, we have that $\frac{1}{n}\sum_{k=0}^{n-1}T^k=\Phi^{-1}\left(\frac{1}{n}\sum_{k=0}^{n-1}S^k\right)\Phi$ with $\left\|\frac{1}{n}\sum_{k=0}^{n-1}S^k\right\|\leq 1$ for all $n\geq 1$. Moreover, it is proved in Theorem 3.2 of \cite{Bena} that $T$ does not have any fixed point. By Theorem \ref{ergo}, we deduce that $T$ is not ergodic. So $A$ is a non-ergodic subset of $C$ and the proof is complete.
\end{proof}

Now we prove the super-version of the previous theorem:

\begin{thm}\label{super-ergodic}
Let $C$ be a closed convex subset of a Banach space $X$. The following assertions are equivalent:
\begin{enumerate}
    \item[(i)] $C$ is SWC;
    \item[(ii)] $C$ is super-ergodic;
    \item[(iii)] any affine isometry from a set f.r. in $C$ into itself is ergodic.
\end{enumerate}
\end{thm}

\begin{proof}
$(i)\implies (ii)$ follows directly from the previous theorem. $(ii)\implies (iii)$ is obvious.

Suppose that $C$ is not SWC and let us show that $(iii)$ does not hold. There exists a free ultrafilter $\mathcal U$ on $\mathbb N$ such that $C^\mathcal U$ is not weakly compact. So there exists a sequence $(x_n)_n\in C^\mathcal U$ without any convergent subsequence. Taking subsequence if necessary, we can suppose that $(x_n)_n$ is a good sequence. Let $Z$ be the spreading model built on $(x_n)_n$ with fundamental sequence $(e_n)_n$ and consider let $T$ be the shift of $Z$.

If $(e_n)_n$ is equivalent to the canonical basis of $l_1$ then the mean sequence $(\frac{1}{n}\sum_{k=0}^{n-1}T^k(e_1))_n$ does not converge since $\frac{1}{n}\sum_{k=0}^{n-1}T^k(e_1)=\frac{1}{n}\sum_{k=0}^{n-1}e_k$. Remember that $\cconv\{e_n\}_n$ is  f.r. in $\cconv\{x_n\}_n\subset C^\mathcal U$ by Corollary \ref{fr_sm} and that $C^\mathcal U$ is f.r. in $C$ by Proposition \ref{fr1}. Since $\cconv\{e_n\}_n$ is not ergodic, it follows that $(ii)$ does not hold.

Now we suppose that $(e_n)_n$ is not equivalent to the canonical basis of $l_1$. To conclude, we are going to show again that the Cesaro mean sequence $(\frac{1}{n}\sum_{k=0}^{n-1}T^k(e_1))_n$ can not converge. Suppose on the contrary that this sequence converges. This implies that $(e_n)_n$ weakly converges (see Proposition 4 p.21 in \cite{Beauzamy}). Since $(e_n)_n$ is not equivalent to the canonical basis of $l_1$, it follows that $(x_n)_n$ weakly converges (by Theorem 3 p.25 in \cite{Beauzamy}) which is a contradiction.
\end{proof}

\begin{thm}
Super-reflexivity is equivalent to super-ergodicity.
\end{thm}

\begin{proof}
Suppose that $X$ is not super-ergodic and let $Y$ be a Banach space f.r. in $X$ which is not ergodic. So $B_Y$ is not ergodic. By corollary \ref{weakly_compact}, it follows that $B_Y$ is not weakly compact. So $Y$ is not reflexive and $X$ is not super-reflexive. Now suppose that $X$ is super-ergodic and let $Y$ be a Banach space which is f.r. in $X$. Then $B_Y$ is f.r. in $B_X$ and then $B_Y$ is ergodic by the previous theorem. So $Y$ is ergodic by Proposition \ref{ergo_Banach}.
\end{proof}

\section{Fixed point property and (super) weak compactness}

The objective of this section is to generalize Theorem 3.6 in \cite{SWC}. The authors proved that a closed bounded convex subset $C$ of a Banach space $X$ is SWC if and only if any affine isometry $T:C\to C$ which can be extended to an affine isometry on $X$ has a fixed point. However, it could exist affine isometries on $C$ without any affine isometric extension to $X$. That is why we propose a intrinsic characterization. 

\begin{defi}
Let $\mathcal C$ be a class of convex mappings. We say that a closed convex bounded subset $C$ of a Banach space $X$ has \textit{the fixed point property} (FPP in short) for $\mathcal C$, if every mapping from $C$ into itself belonging to $\mathcal C$ has a fixed point. We say that $C$ has the \textit{super-FPP} for $\mathcal C$ if any convex set which is f.r. in $C$ has the FPP for $\mathcal C$. Finally, if any closed convex bounded subset of $X$ has the FPP for $\mathcal C$, we say that $X$ has the FPP for $\mathcal C$.
\end{defi}

\begin{lema}
Let $C$ be a bounded closed convex subset of a Banach space $X$. If $T:C\to C$ is an affine continuous mapping then $T$ is weakly continuous.
\end{lema}

\begin{proof}
Let $(x_a)_{a\in A}$ be a net in $C$ that weakly converges to some $x\in C$. Suppose that $(T(x_a))_{a\in A}$ does not weakly converge to $T(x)$. Then there exists a weak open neighborhood $V$ of $T(x)$ and a subnet $(x_b)_{b\in B}$ of $(x_a)_{a\in A}$ such that $T(x_b)\notin V$ for all $b\in B$. We can write $V=\bigcap_{i=1}^pU_i$ with $U_i=\{y\in X\ |\ x^*_i(y-T(x))<\varepsilon\}$. So, by taking another subnet if necessary, we can suppose that there exists $i_0\in\{1,...,p\}$ such that $x_b\notin U_{i_0}$ for all $b\in B$. Since $(x_b)_{b\in B}$ weakly converges to $x$, we have that $x\in\cconv\{x_b\}_{b\in B}$ and then there exists a sequence $(y_n)_{n\in\mathbb N}\subset\co\{x_b\}_{b\in B}$ such that $y_n\to x$. By continuity of $T$, we have that $T(y_n)\to T(x)$. However using that $T$ is affine, the convexity of $U_i^c$ and the fact that $T(x_b)\notin U_{i_0}$ for all $b\in B$, it is easy to see that $T(y_n)\notin U_{i_0}$ for all $n\in\mathbb N$. This is a contradiction and the proof is complete.
\end{proof}

\begin{lema}\label{weak_continuous}
Let $C$ be a convex subset of a Banach space $(X,\|.\|)$ such that $\|.\|^2$ is uniformly convex on $C$. If $D$ is a convex subset of a Banach space $(Y,|.|)$ which is f.r. in $C$, then $|.|^2$ is uniformly convex on $D$.
\end{lema}

\begin{proof}
Define $\delta(t)=\min\{\delta_{\|.\|^2}\left(\frac{t}{2}\right),\delta_{g}(t)\}>0$ where $g(s)=s^2$ for all $s\in\mathbb R$. Let $x,y\in D$. Suppose first that $x$ and $y$ are linearly independent. For all $n\in\mathbb N$, there exist $C_n\subset C$ and an isomorphism $T_n:\text{span}\{x,y\}\to\text{span}(C_n)$ such that $T_n(x),T_n(y)\in C$ and $$\left(1-\frac{1}{n}\right)|z|\leq\|T_n(z)\|\leq\left(1+\frac{1}{n}\right)|z|$$ for all $z\in\text{span}\{x,y\}$. For all $n\geq 2$, it follows that
\begin{align*}
\left(1-\frac{1}{n}\right)^2\left|\frac{x+y}{2}\right|^2&\leq\left\|\frac{T_n(x)+T_n(y)}{2}\right\|^2\\
&\leq\frac{\|T_n(x)\|^2+\|T_n(y)\|^2}{2}-\delta_{\|.\|^2}(\|T_n(x)-T_n(y)\|)\\
&\leq\left(1+\frac{1}{n}\right)^2\frac{|x|^2+|y|^2}{2}-\delta_{\|.\|^2}\left(\frac{|x-y|}{2}\right)
\end{align*}
and letting $n\to\infty$ we obtain that $$\left|\frac{x+y}{2}\right|^2\leq\frac{|x|^2+|y|^2}{2}-\delta_{\|.\|^2}\left(\frac{|x-y|}{2}\right)\leq\frac{|x|^2+|y|^2}{2}-\delta(|x-y|).$$ Now if $x$ and $y$ are linearly dependant, one can easily prove that $$\left|\frac{x+y}{2}\right|^2\leq\frac{|x|^2+|y|^2}{2}-\delta_g(|x-y|)\leq\frac{|x|^2+|y|^2}{2}-\delta(|x-y|)$$ and the proof is complete.
\end{proof}

Let $C$ be a closed convex subset of a Banach space $X$. We recall that a point $x\in C$ is \textit{diametral} if $\sup\{\|x-y\| : y\in D\}=\text{diam}(D)$. We say that $C$ has \textit{normal structure} (see \cite{Kirk}) if any bounded closed convex subset $D$ of $C$ containing more than one point has a point which is not diametral.

\begin{lema}\label{fr_unifconv}
Suppose that $C$ is a convex subset of a Banach space $X$ such that $\|.\|^2$ is uniformly convex on $C$. Then $C$ has normal structure.
\end{lema}

\begin{proof}
Let $D$ be a bounded closed convex subset of $C$ and let $d=\text{diam}(D)$. Fix $x,y\in D$ two distinct point and let us show that $\frac{x+y}{2}$ is not a diametral point of $D$. Suppose that it is diametral. Then for all $n\in\mathbb N$ there exists $x_n\in D$ such that $\left\|\frac{x+y}{2}-x_n\right\|^2>d^2-\frac{1}{n}.$ It follows that:
\begin{align*}
d^2-\frac{1}{n}&<\left\|\frac{x+y}{2}-x_n\right\|^2\\
&=2\left\|\frac{1}{2}\left(\frac{x-x_n}{2}+\frac{y-x_n}{2}\right)\right\|^2\\
&\leq\left\|\frac{x-x_n}{2}\right\|^2+\left\|\frac{y-x_n}{2}\right\|^2-\delta_{\|.\|^2}(\|x-y\|)\\
&\leq d^2-\delta_{\|.\|^2}(\|x-y\|)
\end{align*}
and we obtain a contradiction by letting $n\to\infty$.
\end{proof}

\begin{thm}\label{SWC_FPP}
Let $C$ be a bounded closed convex subset of a Banach space $X$. The following assertions are equivalent:
\begin{enumerate}
    \item[(i)] $C$ is SWC;
    \item[(ii)] $C$ has the super-(FPP for affine isometries);
    \item[(iii)] $C$ has the super-(FPP for continuous affine mappings);
    \item[(iv)] there exists an equivalent norm on $X$ such that $C$ the super-(FPP for non-expansive mappings).
\end{enumerate}
\end{thm}

\begin{proof}
Any continuous affine self-mapping of a closed convex
set is weakly continuous by Lemma \ref{weak_continuous}. It follows that $(i)\implies (iii)$ by Schauder-Tychonoff theorem. The implication $(iii)\implies (ii)$ is obvious.

$(ii)\implies (i)$ Just follow the proof of $(ii)\implies (i)$ in Theorem \ref{super-ergodic}.


$(i)\implies (iv)$ By Theorem \ref{unif_convex_renorm}, there exists an equivalent $|.|$ norm on $X$ such that $|.|^2$ is uniformly convex on $C$. Let us show that $(C,|.|)$ has the super-(FPP for non-expansive mappings). Let $D$ be any convex set f.r. in $(C,|.|)$. By the two previous lemmas,  $D$ has normal structure. By Kirk's theorem (see \cite{Kirk}), it follows that $D$ has the FPP for non-expansive mapping.

$(iv)\implies (i)$ Under this new norm, $C$ has the super-(FPP for affine isometries) and we deduce that $C$ is SWC thanks to the implication $(ii)\implies (i)$.
\end{proof}


\begin{rema}
In general, it is not true that a convex SWC set has the super-(FPP for isometries). In fact, Alspach constructed a weakly compact set $K$ (and then SWC, see the next part) of $L_1[0,1]$ and an isometry $T:K\to K$ without any fixed point (see \cite{Alspach}).
\end{rema}

\begin{rema}
The implication $(i)\implies (ii)$ can also be proved directly using the existence of lower semi-continuous uniformly convex functions on $C$. The ideas of this construction can be found in \cite{unifconvex} and we refer the reader to this document for the definitions of the different notions used in the following argument. Let $T:C\to C$ be any isometric affine mapping. For $\varepsilon>0$, we define $f_\varepsilon(x)$ as the height of the tallest $\varepsilon$-separated dyadic tree with root $x$. Note that gluing trees, it is easy to see that $f_\varepsilon$ fulfills that $$f_\varepsilon\left(\frac{x+y}{2}\right)\geq\min\{f_\varepsilon(x),f_\varepsilon(y)\}+1$$ whenever $\|x-y\|\geq\varepsilon$, i.e. $f_\varepsilon$ is a $\varepsilon$-quasi concave function. Since a $\varepsilon$-separated dyadic tree with root $x$ gives a $\varepsilon$-separated dyadic tree with root $T(x)$ through $T$, we have that $f_\varepsilon\circ T\geq f_\varepsilon$. We easily deduce that the function $h_\varepsilon:=3^{-f_\varepsilon}$ is $\varepsilon$-uniformly convex and verifies that $h_\varepsilon\circ T\leq h_\varepsilon$. We deduce that the closed convex envelope $g_\varepsilon=\cconv(h_\varepsilon)$ of $h_\varepsilon$ is $2\varepsilon$-uniformly convex convex and lower semi-continuous. Moreover, note that $g_\varepsilon\circ T\leq g_\varepsilon$. In fact, we have that $g_\varepsilon\circ T=\cconv(h_\varepsilon)\circ T\leq h_\varepsilon\circ T\leq h_\varepsilon$ where $g_\varepsilon\circ T$ is convex lower semi-continuous and then $g_\varepsilon\circ T\leq \cconv(h_\varepsilon)=g_\varepsilon$. The closed envelope $g$ of the function $\sum_{n\geq 1}\frac{1}{2^n\|g_{1/n}\|_\infty}g_{1/n}$ is uniformly convex, lower semi-continuous and verifies that $g\circ T\leq g$. The function $g$ is convex lower semi-continuous (thus convex $w$-lower semi-continuous) on a weakly compact set and then reaches his minimum on $C$ at some point $x\in C$. Since $g$ is uniformly convex, this miminum is unique. Moreover, we have that $g(T(x))\leq g(x)$. It follows that $T(x)=x$.
\end{rema}

\section{Application to S$^2$WCG Banach spaces}

The following definition was in \cite{S2WCG}:

\begin{defi}
A Banach space $X$ is said to be \textit{strongly super weakly compactly generated} (in short, S$^{\,2}$WCG) if there is a SWC set $K \subset X$ that strongly generates any weakly compact set $H \subset X$.
\end{defi}

Note that if $X$ is S$^{\,2}$WCG then any weakly compact subset of $X$ is SWC (see Theorem 4.1 in \cite{Cheng}). A fundamental example of a S$^{\,2}$WCG space is $L^1(\Omega,\mathcal A,\mu,X)$ where $(\Omega,\mathcal A,\mu)$ is a finite measure space and $X$ is a superreflexive Banach space (see \cite{S2WCG} for the proof).\\




Theorem 4.6 of \cite{SWC} is obtained as an easy consequence of the previous results:

\begin{prop}
Let $X$ be a S$^2$WCG Banach space. Then $X$ admits an equivalent norm such that any weakly compact convex subset has the FPP for non-expansive mappings.
\end{prop}

\begin{proof}
Let $K$ be a SWC absolutely convex set that strongly generates $X$. Consider that $X$ is endowed with the norm given by Theorem \ref{unif_convex_renorm}. It follows that the square of the norm is uniformly convex on any weakly compact subset of $X$ and then any weakly compact subset has normal structure by Lemma \ref{fr_unifconv}. The conclusion is obtained thanks to Kirk's theorem.
\end{proof}

Combining the results of the previous parts, we obtain:

\begin{prop}
Let $C$ be a closed convex subset of a S$^2$WCG Banach space $X$. The following assertions are equivalent:
\begin{enumerate}
    \item[(i)] $C$ is weakly compact;
    \item[(ii)] $C$ is SWC;
    \item[(iii)] $C$ is superergodic;
    \item[(iv)] any closed convex subset of $C$ is ergodic;
    \item[(v)] any closed convex subset of $C$ has the FPP for continuous affine mappings.
\end{enumerate}
\end{prop}

\begin{proof}
$(i)\implies (ii)$ follows from the fact that $X$ is S$^2$WCG. $(ii)\implies (iii)$ follows from Theorem \ref{super-ergodic}. $(iii)\implies (iv)$ is obvious. $(iv)\implies (v)$ follows theorem \ref{ergo}. $(v)\implies (i)$ is Theorem 3.2 in \cite{Bena}.
\end{proof}

\begin{prop}
Let $Y$ be a subspace of a S$^2$WCG Banach space $X$. The following assertions are equivalent:
\begin{enumerate}
    \item[(i)] $Y$ is reflexive;
    \item[(ii)] $Y$ is superreflexive;
    \item[(iii)]$Y$ is super-ergodic;
    \item[(iv)] $Y$ is ergodic;
    \item[(iv)] $Y$ has the FPP for continuous affine mappings.
\end{enumerate}
\end{prop}

\begin{proof}
It follows directly from the previous results.
\end{proof}

\begin{rema}
In the case of $L^1(\Omega,\mathcal A,\mu,\mathbb R)$, the previous corollary can be considerably improved. In that case, $Y$ is reflexive if and only if $Y$ has the FPP for non-expansive mappings. One implication is due to Maurey (Theorem 1 in \cite{Maurey_L1}) and the other one is due to Dowling and Lennard (Theorem 1.4 in \cite{Dowling}).
\end{rema}

\section{A remark on the M-FPP}

\begin{defi}\label{M-property}
Let $(\mathcal P)$ be a property of Banach spaces. We say that a Banach space $X$ has the property M-$(\mathcal P)$ if any spreading model of $X$ has $(\mathcal P)$.
\end{defi}

It is worth noting that it does not imply that $X$ has $(\mathcal P)$ in general. The notion of M-property has been introduced by Beauzamy in \cite{Beauzamy} (see Chapter 5). In this book, the author claims that there does not exist any characterization of the M-reflexivity. As far as we know, this question is still opened.\\

Remember that a Banach space $X$ has the \textit{Banach-Saks property} (resp. \textit{alternate Banach-Saks property}) if any bounded sequence $(x_n)_n\subset X$ admits a subsequence $(x_{\phi(n)})_n$ such that $\left(\sum_{k=1}^nx_{\phi(k)}\right)_n$ (resp. $\left(\sum_{k=1}^n(-1)^kx_{\phi(k)}\right)_n$) converges. For more informations about these properties and its links with spreading models, we refer the reader to \cite{Beauzamy}. \\

We start with the following lemma:

\begin{lema}
The M-(FPP for affine isometries) implies the alternate Banach-Saks property.
\end{lema}

\begin{proof}
Let $X$ be a Banach space and suppose that $X$ does not have the alternate Banach-Saks property. By Beauzami's theorem (see Theorem 5 p.47 in \cite{Beauzamy}, it follows that $X$ has a spreading model $Z$ isomorphic to $l_1$. Then the fundamental sequence $(e_n)_n$ of $Z$ is equivalent to the canonical basis of $l_1$. In particular, $(e_n)_n$ is not weakly convergent. By the previous proposition, we deduce that $\cconv\{e_n\}_n$ and thus $Z$ do not have the FPP for affine isometries. 
\end{proof}

It is well-known that the fixed point property does not imply reflexivity. In fact, $l_1$ can be renormed to have the FPP for non-expansive mappings (see \cite{Lin}). However the M-FPP implies reflexivity. More precisely we have that:

\begin{thm}\label{M-FPP}
The M-(FPP for affine isometries) implies the Banach-Saks property.
\end{thm}

\begin{proof}
Let $X$ be a Banach space with the M-(FPP for affine isometries). By the previous lemma, we already know that $X$ has the alternate Banach-Saks property. Now we prove that $X$ is reflexive. By contradiction, suppose that $X$ is not reflexive. There exists a bounded sequence $(x_n)_n$ without any weakly convergent subsequence. By taking subsequence if necessary, we can suppose that $(x_n)_n$ is a good sequence generating a spreding model $Z$ with fundamental sequence $(e_n)_n$. Since $X$ has the alternate Banach-Saks property, $(e_n)_n$ is not equivalent to the canonical basis of $l_1$. Since $(x_n)_n$ is not weakly convergent, it follows that $(e_n)_n$ is not weakly convergent (see Theorem 3 p.25 in \cite{Beauzamy}). By the previous proposition, we conclude that $Z$ can not have the FPP for affine isometries, which is a contradiction. It follows that $X$ is reflexive and has the alternate Banach-Saks property and then has the Banach-Saks property (see Proposition II-4-1 in \cite{Beauzamy}).
\end{proof}

Before the next theorem, we recall some well-known facts. First the Banach-Saks property implies reflexivity (see \cite{BS_reflexive} for a proof or note that $\ell_1$ does not have the Banach-Saks property and apply Rosenthal's $\ell_1$-theorem). Moreover, any superreflexive Banach space admits an uniformly convex norm by Enflo's theorem. Then it follows that superreflexive Banach spaces have the Banach-Saks property by Kakutani's theorem (see \cite{Kakutani}). In particular, superreflexivity is equivalent to the super-Banach-Saks property.

Maurey proved that any isometry $T:C\to C$ on a closed convex subset $C$ of a superreflexive Banach space has a fixed point (see for example Theorem F p.112 in \cite{Maurey}). Any continuous affine mapping $T:C\to C$ also enjoys this property:

\begin{thm}
Let $X$ be a Banach space. The following assertions are equivalent:
\begin{enumerate}
    \item[(i)] $X$ is superreflexive;
    \item[(ii)] $X$ has the super-(FPP for affine isometries);
    \item[(iii)] $X$ has the super-(FPP for continuous affine mappings);
    \item[(iv)] $X$ has the super-(FPP for isometries).
\end{enumerate}
\end{thm}

\begin{proof}
By the previous theorem, we have that the super-(FPP for affine isometries) implies the Banach-Saks property. Since super-reflexivity is equivalent to the super-Banach-Saks property, we obtain $(ii)\implies (i)$. We have that $(i)\implies (iii)$ by Schauder-Tychonoff theorem and $(iii)\implies (ii)$ is obvious. $(i)\implies (iv)$ is Maurey's theorem. Since $(iv)\implies (ii)$ is obvious, the proof is complete.
\end{proof}

\textbf{Acknowledgements:} 
The first named author is also very grateful to G. Lancien and G. Godefroy for fruitful conversations.

The authors were supported by grants of Ministerio de Ciencia e Inovaci\'on PID2021-122126NB-C32 and Fundación Séneca - ACyT Región de Murcia project 21955/PI/22. The first named author was also supported by MICINN 2018 FPI fellowship with reference PRE2018-083703, associated to grant MTM2017-83262-C2-2-P.

\end{document}